\documentclass{amsart}
\usepackage[all]{xy}
\usepackage{amsmath}
\usepackage{amssymb}
\usepackage{amsthm}
\usepackage{amscd}

\usepackage{geometry}
\usepackage{graphicx}
\usepackage{float}
\usepackage{wrapfig}

\newtheorem{theorem}{Theorem}

\numberwithin{theorem}{section}

\DeclareMathOperator*{\esssup}{ess\,sup}

\usepackage{hyperref}

\begin{document}
\title{Bounded operators on mixed norm Lebesgue spaces}

\author{Nikita Evseev}
\author{Alexander Menovschikov}
\address{
Novosibirsk State University\\
 1 Pirogov st, Novosibirsk, 630090
}
\address{
Sobolev Institute of Mathematics\\
4 Acad. Koptyug avenue, 630090 Novosibirsk Russia
}

\address{
Peoples Friendship University of Russia\\
117198 Moscow, Russia
}

\email{evseev@math.nsc.ru}
\email{menovschikov@math.nsc.ru}

\thanks{The publication was supported by the Ministry of Education and Science of the Russian Federation (Project number 1.3087.2017/4.6)}

\date{\today}
\maketitle

\begin{abstract}

We study two classes of bounded operators on mixed norm Lebesgue spaces, namely composition operators and product operators.
A complete description of bounded composition operators on mixed norm Lebesgue spaces are given.
For a certain class of integral operators, we provide sufficient conditions for boundedness.
We conclude by applying the developed technique to the investigation of Hardy-Steklov type operators.

\end{abstract}

\section{Introduction}
A mixed norm Lebesgue space consists of measurable multivariable functions with 
a norm defined in terms of different iteratively calculated $L_p$ norms. 
These spaces naturally arise when dealing with partial differential
equations e.g.\ in obtaining a priori estimates  \cite{Inv}. 
Moreover a mixed norm Lebesgue space $L_{p,q}$ represents an example of $L_p$ of Banach space valued functions. 
Due to the fact that a broad range of problems naturally lead to mixed norm Lebesgue spaces, they have been studied intensively in recent years.
In this work, we study the boundedness of both composition and multiplication operators, as well as the boundedness of some classes of integral operators.

A composition operator can be defined on any function space. 
Studying properties of this operator allows us to ascertain those changes of variables that map one given function class to another. 
For mixed norm Lebesgue spaces, composition operators were abandoned except for the case of holomorphic functions (e.g.\ \cite{St}).  
Herein, we intend to begin exploration of composition operators on mixed norm Lebesgue spaces.
 

Hardy type operators on mixed norm Lebesgue spaces have been studied in a series of papers \cite{JJG, FGJ}.
The basis for this research was laid in \cite{AK}. 
Similar results were established for a wider class of integral operators in a recent work \cite{GS}. 
In this paper, we present 
generalizations of the results in article \cite{GS} and, as a consequence, we obtain the boundedness of the Hardy-Steklov type  operators.

\section{Mixed Lebesgue Space}

Let $(X_i,\mu_i)$, $i=1,\dots n$, be $\sigma$-finite measurable spaces and 
$\Omega$ be a measurable set from 
$\prod\limits_{i=1}^n X_i$.
Let $1\leq p_i<\infty$ and $P=(p_1,\cdots, p_n)$. 
Define a mixed norm space $L_P(\Omega)$ as a set of measurable functions $f(x)$ with
finite norm
$$ 
||f||_{L_P(\Omega)} = 
\bigg(\int\limits_{X_1}\bigg(\cdots 
\bigg(\int\limits_{X_n} |f(x)|^{p_n} \chi_{\Omega}(x) \, d\mu_n\bigg)^{\frac{p_{n-1}}{p_n}}
\cdots\, \bigg)^{\frac{p_1}{p_2}}  \, d\mu_1\bigg)^{\frac{1}{p_1}}.
$$

Mixed norm Lebesgue spaces are usually defined on a cartesian product and not on an arbitrary domain. 
However, we follow the approach of \cite{Be} and chose a more general definition which helps investigate finer properties.

It is easy to see that a mixed norm Lebesgue space is an example of a nonrearrangement invariant space.
Thus functions from $L_P$ are virtually multivariable i.e.\  such a function can not be represented by its distribution.
Properties of multivariate rearrangements can be found in \cite{Bl}.   

We adopt the following notations:
denote by
$\Omega_{x_1}$ the set of $\tilde x = (x_2,\dots,x_{n})
\in \prod\limits_{i=2}^{n}X_i$ such that $(x_1, \tilde x)$ are in $\Omega$,
then we can write the norm in the form
\begin{equation}\label{norm2}
\|f\|_{L_P(\Omega)} = \bigg(\int\limits_{X_1} \|f\|_{L_{\widetilde{P}(\Omega_{x_1})}}^{p_1} \, d\mu_1\bigg)^{\frac{1}{p_1}},
\end{equation}
where $\widetilde P = (p_2,\dots,p_{n})$ as well define the projection $\pi_i(\Omega)$ of $\Omega$ on space $X_i$.

\section{Composition operator}\label{composition}
For an investigation of composition operators one needs more specific structures on $X_i$.
In particular, we will consider spaces of homogeneous type, 
which are metric measurable spaces with specific relations between the metric and measure
\cite{VU}.

Now let $(X_i, d_i,\mu_i)$, $i=1,\dots n$, be spaces of the homogeneous type. 
Consider another sequence of homogeneous type spaces  $(Y_i,\rho_i,\nu_i)$
and a measurable mapping $\varphi\colon\Omega\to \Omega'$, where  $\Omega'\subset \prod\limits_{i=1}^n Y_i$. 

The mapping $\varphi$ induces a bounded composition operator on mixed norm Lebesgue spaces
$$
C_{\varphi}:L_{P}(\Omega')\to L_{P}(\Omega) \quad\text{ by the rule }\quad (C_{\varphi} f) = f\circ\varphi
$$
whenever $f\circ\varphi \in L_{P}(\Omega)$ and 
$\|C_{\varphi} f\|_{L_{P}(\Omega)}\leq K\|f\|_{L_{P}(\Omega')}$
for every function $f\in L_{P}(\Omega')$, the constant $K$ is independent of choice of $f$.

\textit{Our goal is to find necessary and sufficient conditions on the mapping $\varphi$ under which the composition operator $C_{\varphi}$ is bounded.}

The full description in the case of composition operator on Lebesgue spaces was given in \cite{VU}, 
see also an exhaustive survey on the topic in the book \cite{SM}. 
From \cite[Theorem 4]{VU} it follows that the norm of the bounded composition operator
\begin{equation}\label{norm-lp}
\|C_\varphi\|_{L_p\to L_p} = \esssup\limits_{y \in \Omega'}J^{\frac{1}{p}}_{\varphi^{-1}}(y).
\end{equation}

A mixed norm blending of variables can lead to expulsion from the function class. 
For example consider the mapping $\varphi(x_1,x_2) = (x_2, -x_1)$, which is a rotation in $\mathbf R^2$.
Take the function $f(y_1, y_2) = \frac{1}{(1+|y_1|)\sqrt{1+|y_2|}}\in L_{2,3}(\mathbf R^2)$, on the other hand 
a composition $f\circ\varphi(x_1,x_2) = \frac{1}{(1+|x_2|)\sqrt{1+|x_1|}}$ does not belong to $L_{2,3}(\mathbf R^2)$.
Therefore it is natural to consider changes of variables that preserve the privileged role of 
the first variable over the second and so forth.


Taking into account the above arguments we study changes of variables in the form 
\begin{equation}\label{themap}
\varphi(x) = (\psi_1(x_1), \psi_2(x_1,x_2), \dots, \psi_n(x_1,\dots, x_n ))
\end{equation}
with all $\psi_i$ being injective, except for possible the last $\psi_n$.

Finally, for a composition to be well defined we must assume that 
$\psi_i(x_1,\dots, x_{i-1}, \cdot)$ enjoys the Luzin $\mathcal N^{-1}$-property
for $\mu_1\times\dots\times\mu_{i-1}$-a.e. $(x_1,\dots,x_{i-1})\in \prod\limits_{j=1}^{i} X_j$.
Due to  Luzin $\mathcal N^{-1}$-properties the measure 
$\mu_i\circ\psi_i^{-1}(y_1,\dots,y_{i-1},\cdot)$ is absolutely continuous with respect to $\nu_j$ and
 for $\nu_1\times\dots\times\nu_{i-1}$-a.e. $(y_1,\dots,y_{i-1})\in \prod\limits_{j=1}^{i} Y_j$. 
Thus there are Radon--Nikodym derivatives 
$J(\psi^{-1}_i(y_1,\dots,y_{i-1},\cdot); y_i) : \pi_i(\Omega')\to \mathbf R$ (see figure \ref{fig:1} below)  such that
$$
\mu_i(\psi^{-1}_i(y_1,\dots,y_{i-1},E)) 
= \int\limits_{E}J(\psi^{-1}_i(y_1,\dots,y_{i-1},\cdot); y_i)\, d\nu_i 
$$
for any measurable set $E\in Y_i$ 
and for $\nu_1\times\dots\times\nu_{i-1}$-a.e. $(y_1,\dots,y_{i-1})\in \prod\limits_{j=1}^{i} Y_j$.
Thus the change of variables formula for with mapping $\psi_1$  is written as 
\begin{equation}\label{eq:chf1}
\int\limits_{\psi^{-1}_1(E)}g(\psi_1(x_1))\, d\mu_1 = \int\limits_Eg(y_1)J(\psi^{-1}_1(\cdot); y_1)\, d\nu.
\end{equation}
Whereas for mapping $\psi_2$ the formula turns to
$$
\int\limits_{\psi^{-1}_1(x_1,D)}g(\psi_1(x_1),\psi_2(x_1,x_2))\, d\mu_1 = \int\limits_Dg(y_1)J(\psi^{-1}_2(y_1,\cdot); y_2)\, d\nu 
\quad \text{a.e} x_1\in\pi_1(\Omega) \text{ and } y_1=\psi_1(x_1), 
$$
for measurable set $D\subset\pi_2(\Omega'_{y_1})$ and  a. e.\ $x_1\in\pi_1(\Omega), y_1=\psi_1(x_1)$.

\begin{figure}[H]
\includegraphics[width=\linewidth]{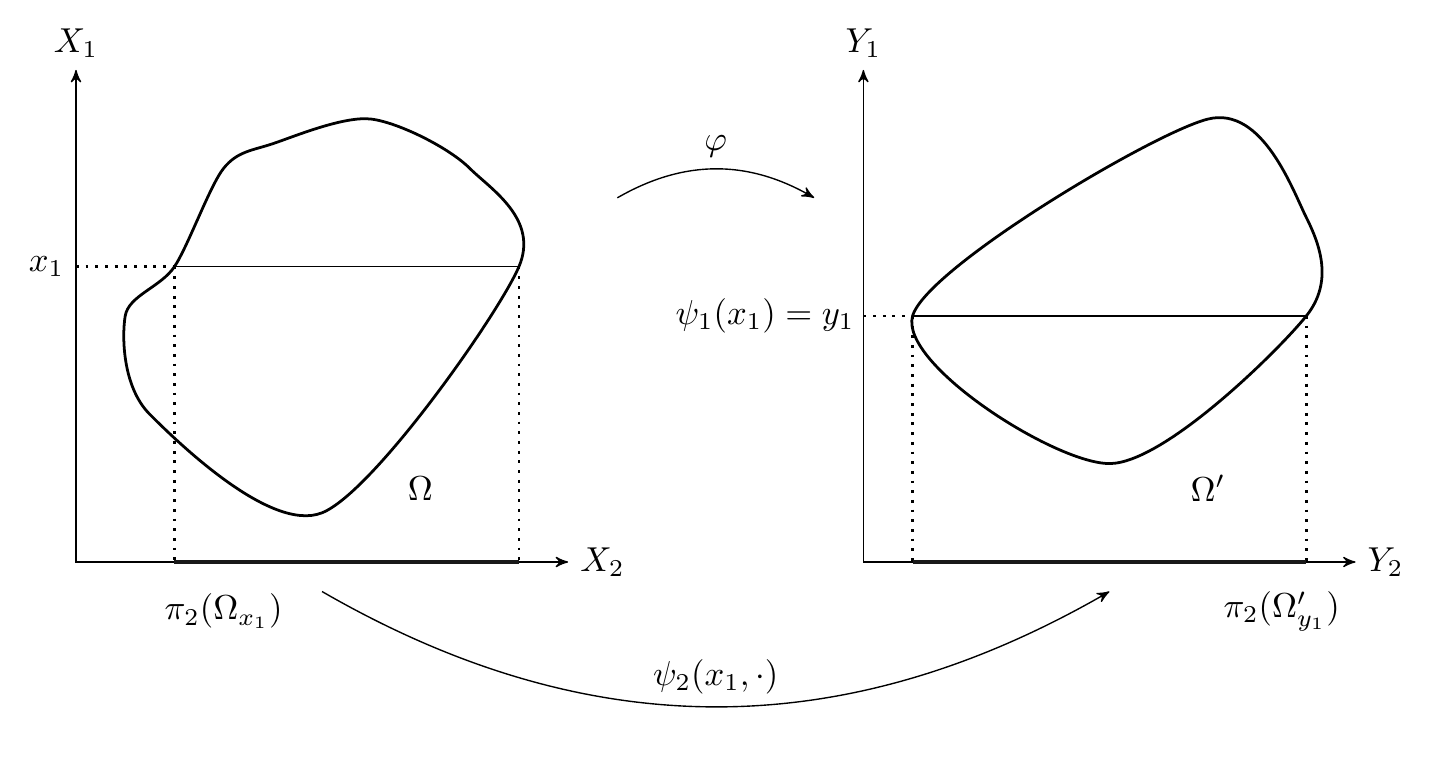}
\caption{The mapping $\varphi$ acts by layers.}
\label{fig:1}
\end{figure}

Now we are ready to state and prove the following result
\begin{theorem}\label{theorem:boundedLpq}
A measurable mapping $\varphi:\Omega\to \Omega'$ of the form \eqref{themap} induces a bounded operator
$C_{\varphi}:L_{P}(\Omega')\to L_{P}(\Omega)$, if and only if
\begin{equation}\label{eq1:theorem:boundedLpq}
\esssup\limits_{\Omega'}\bigg\{\prod\limits_{i=1}^nJ(\psi^{-1}_i(y_1,\dots,y_{i-1},\cdot); y_i)^{\frac{1}{p_i}}\bigg\}
 < \infty.
\end{equation}
And the norm of the operator
$$
\|C_\varphi\| = \esssup\limits_{\Omega'}\bigg\{\prod\limits_{i=1}^nJ(\psi^{-1}_i(y_1,\dots,y_{i-1},\cdot); y_i)^{\frac{1}{p_i}}\bigg\}.
$$
\end{theorem}

\begin{proof}
We prove the necessity by induction.
We start with the base case of $n=1$ i.e. \eqref{eq1:theorem:boundedLpq} is a consequence of \cite[Theorem 4]{VU}.
Then, the inductive step is as follows.
Suppose the theorem is proved for some $n$. 
Prove that \eqref{eq1:theorem:boundedLpq} holds for $n+1$.

If operator $C_{\varphi}:L_{P}(\Omega')\to L_{P}(\Omega)$ is bounded then
$$
\|f\circ\varphi\|_{L_{P}(\Omega)} 
\leq K \|f(y)\|_{L_{P}(\Omega')}.
$$


Take a ball $B=B(\stackrel oy_1,\rho)\subset Y_1$ and an arbitrary family $\{g_{y_1}(\tilde{y})\}$ which consists
of such functions $g_{y_1}(\tilde y)\in L_{\widetilde P}(\Omega'_{y_1})$ and 
$\|g_{y_1} \|_{L_{\widetilde P}(\Omega'_{y_1})} = 1$, that $\tilde y = (y_2,\dots,y_{n+1})$.
Then for function $f(y_1,\tilde y) = \chi_B(y_1)g_{y_1}(\tilde y)$
the composition $f\circ\varphi(x) = \chi_B(\psi_1(x_1))g_{\psi_1(x_1)}(\tilde \varphi(x))$,
where $\tilde\varphi(x) = (\psi_2(x_1,x_2), \dots, \psi_n(x_1,\dots, x_n ))$.
Taking into account \eqref{norm2} the above inequality takes the form 
$$
\bigg(\int\limits_{\psi_{1}^{-1}(B)} \|g_{\psi_1(x_1)}\circ\tilde\varphi(x_1,\cdot)\|_{L_{\tilde P}(\Omega_{x_1})}^{p_1} \, d\mu_1\bigg)^{\frac{1}{p_1}}
\leq K (\nu_1(B))^{\frac{1}{p_1}}.
$$
Or, because of the arbitrariness of the choice of families
$$
\bigg(\int\limits_{\psi_1^{-1}(B)} 
\sup\|g_{\psi_1(x_1)}\circ \tilde\varphi(x_1,\cdot) \|^{p_1}_{L_{\tilde P}(\Omega_{x_1})}
\, d\mu_1\bigg)^{\frac{1}{p_1}}
\leq K (\nu_1(B))^{\frac{1}{p_1}},
$$
where the supremum is taken over all functions with unit norm. 
Note that the above supremum is the norm of composition operator 
$C_{\tilde\varphi(x_1,\cdot)}:L_{\tilde P}(\Omega'_{\psi_1(x_1)})\to L_{\tilde P}(\Omega_{x_1})$,
which is induced by the mapping $\tilde\varphi(x_1,\cdot)$ for each fixed $x_1$.
Using the induction hypothesis 
$$
\|C_{\tilde\varphi(x_1,\cdot)}\| 
= \esssup\limits_{\Omega'_{\psi_1(x_1)}}\bigg\{\prod\limits_{i=2}^{n+1}J(\psi^{-1}_i(\psi_1(x_1),\dots,y_{i-1},\cdot); y_i)^{\frac{1}{p_i}}\bigg\}
\quad\text{ for a. e. } x_1\in\Omega\cap X_1.
$$

Applying the change of variables formula we deduce
$$
\bigg(\int\limits_{B} 
\esssup\limits_{\Omega'_{y_1}}\bigg\{\prod\limits_{i=2}^{n+1}J(\psi^{-1}_i(y_1,\dots,y_{i-1},\cdot); y_i)^{\frac{1}{p_i}}\bigg\}^{p_1} 
J(\psi_1^{-1}(\cdot); y_1)\, d\nu_1\bigg)^{\frac{1}{p_1}}
\leq K (\nu_1(B))^{\frac{1}{p_1}},
$$

Dividing by $(\nu_1(B))^{\frac{1}{p_1}}$ and applying the Lebesgue theorem we obtain
$$
\esssup\limits_{\Omega'}\bigg\{\prod\limits_{i=1}^{n+1}J(\psi^{-1}_i(y_1,\dots,y_{i-1},\cdot); y_i)^{\frac{1}{p_i}}\bigg\}
 \leq K.
$$
Since the inductive step has been performed then, by mathematical induction, the statement of the theorem holds.

\textit{Sufficiency.}  Let \eqref{eq1:theorem:boundedLpq} hold.  
Again applying the change of variables formulas we arrive at the desired inequality 
\begin{multline*}
\|f\circ\varphi\|_{L_{P}(\Omega)} 
=\bigg(\int\limits_{X_1}\bigg(\cdots 
\bigg(\int\limits_{X_n} |f(\varphi(x))|^{p_n} \chi_{\Omega}(\varphi(x))
 \, d\mu_n\bigg)^{\frac{1}{p_n}}
\cdots\, \bigg)^{p_1}  \, d\mu_1\bigg)^{\frac{1}{p_1}}\\
=\bigg(\int\limits_{Y_1}\bigg(\cdots 
\bigg(\int\limits_{Y_n} |f(y)|^{p_n} \chi_{\Omega'}(y)
J(\psi^{-1}_n(y_1,\dots,y_{n-1},\cdot); y_n) \, d\nu_n\bigg)^{\frac{1}{p_n}}
\cdots\, \bigg)^{p_1}  J(\psi^{-1}_1(\cdot); y_1)\, d\nu_1\bigg)^{\frac{1}{p_1}}\\
\leq \esssup\limits_{\Omega'}\bigg\{\prod\limits_{i=1}^{n+1}J(\psi^{-1}_i(y_1,\dots,y_{i-1},\cdot); y_i)^{\frac{1}{p_i}}\bigg\}\|f\|_{L_{P}(\Omega')}.
\end{multline*}
\end{proof}

\section{Product operator}

As noted in the introduction, Hardy type operators were one of the first class of integral operators to be considered on mixed norm Lebesgue spaces. 
We start this section by obtaining Hardy's inequality in $L_{P}$.
Furthermore, we define a more general class of integral operators, namely product operators, which includes the Hardy operator
and Hardy-Steklov type operators.
 
For our purposes, we will use the Minkowski's integral inequality, which has the following form:
\begin{equation}\label{eq:minkowski}
\bigg(\int\limits_{X_2}\bigg | \int\limits_{X_1} f(x_1,x_2) \, d\mu_1 \bigg |^p \, d\mu_2 \bigg)^{\frac{1}{p}} \leq \int\limits_{X_1} \bigg(\int\limits_{X_2} |f(x_1,x_2)|^p \, d\mu_2 \bigg )^{\frac{1}{p}} \, d\mu_1,
\end{equation}
where $(X_1, \mu_1)$ and $(X_2, \mu_2)$ are two $\sigma$-finite measure spaces and $f : X_1 \times X_2 \to \mathbf{R}$ is measurable.

\subsection{Hardy inequality}

The classical Hardy inequality states if $f$ is an integrable function with non-negative values, then

\begin{equation}\label{eq:hardy}
 \int\limits_{0}\limits^{\infty} \bigg(\frac{1}{x}\int\limits_{0}\limits^{x} f(y) \, dy \bigg)^p \, dx \leq \bigg( \frac{p}{p-1} \bigg)^p \bigg( \int\limits_{0}\limits^{\infty} f(x)^p \, dx \bigg),
\end{equation} 

or in another form 

$$
\|Hf\|_{L_p} \leq \frac{p}{p-1} \|f\|_{L_p},
$$ 

where $H: L_p \to L_p $ is a Hardy operator defined by the rule $Hf(x) = \frac{1}{x}\int\limits_{0}\limits^{x} f(y) \, dy $. 
In \cite{HL} it is shown that the constant is the best possible one. 
Therefore, the norm of Hardy operator is equal to $\frac{p}{p-1}$.

In what follows we will need Hardy's inequality for mixed norm Lebesgue spaces. 
First of all, one defines the operator
$$
H_n f (x) = \frac{1}{x_1\cdot\dots\cdot x_n}\int\limits_{0}^{x_1}\cdots 
\int\limits_{0}^{x_n} f(y) \, dy_n
\cdots\,  dy_1.
$$

In particular, the two dimensional operator 
$H_2: L_{p_1,p_2} \to L_{p_1,p_2} $ is defined by the rule 
$$
H_2 f (x_1,x_2) = \frac{1}{x_1}\int\limits_{0}\limits^{x_1} \frac{1}{x_2}\int\limits_{0}\limits^{x_2} f(y_1, y_2) \, dy_2 \, dy_1.
$$

Upon using the change of variables formula and inequality \eqref{eq:hardy}, we derive  the analogue of Hardy's inequality

\begin{multline*}
 \bigg( \int\limits_{0}\limits^{\infty} \bigg( \int\limits_{0}\limits^{\infty} \bigg( \frac{1}{x_1}\int\limits_{0}\limits^{x_1} \frac{1}{x_2}\int\limits_{0}\limits^{x_2} f(y_1, y_2) \, dy_2 \, dy_1 \bigg)^{p_2} \, dx \bigg)^{\frac{p_1}{p_2}} \, dx_1\bigg)^{\frac{1}{p_1}}\\ 
 = \bigg( \int\limits_{0}\limits^{\infty} \bigg( \int\limits_{0}\limits^{\infty} \bigg( \int\limits_{0}\limits^{1} \int\limits_{0}\limits^{1} f(y_1x_1, y_2x_2) \, dy_2 \, dy_1 \bigg)^{p_2} \, dx_2 \bigg)^{\frac{p_1}{p_2}} \, dx_1\bigg)^{\frac{1}{p_1}} \\
 \leq \int\limits_{0}\limits^{1} \int\limits_{0}\limits^{1} \bigg( \int\limits_{0}\limits^{\infty} \bigg( \int\limits_{0}\limits^{\infty} f(y_1x_1, y_2x_2)^{p_2} \, dx_2 \bigg)^{\frac{p_1}{p_2}} \, dx_1\bigg)^{\frac{1}{p_1}} \, dy_2 \, dy_1\\ 
 = \int\limits_{0}\limits^{1} \int\limits_{0}\limits^{1} \bigg( \int\limits_{0}\limits^{\infty} \bigg( \int\limits_{0}\limits^{\infty} f(x_1, x_2)^{p_2} \, \frac{dx_2}{y_2} \bigg)^{\frac{p_1}{p_2}} \, \frac{dx_1}{y_1} \bigg)^{\frac{1}{p_1}} \, dy_2 \, dy_1 \\
= \frac{p_1}{p_1-1} \frac{p_2}{p_2-1} \bigg(\int\limits_{0}\limits^{\infty} \bigg(\int\limits_{0}\limits^{\infty} f(x_1,x_2)^{p_2}  \, dx_2 \bigg)^{\frac{p_1}{p_2}} \, dx_1 \bigg)^{\frac{1}{p_1}}
\end{multline*}

As in the case of the classical Hardy inequality this statement is equivalent to 
$$
 \|H_2 f\|_{L_{p_1,p_2}} \leq \frac{p_1}{p_1-1} \frac{p_2}{p_2-1} \|f\|_{L_{p_1,p_2}}. 
$$

For arbitary $n$
$$
 \|H_n f\|_{L_{P}} \leq  \prod\limits_{i=1}^n\frac{p_i}{p_i-1} \|f\|_{L_{P}}. 
$$

\subsection{Product operator.}
Let $(X_i,\mu_i)$ and $(Y_i, \nu_i)$, $i = 1,\dots n$ be $\sigma$-finite measurable spaces
and $\Omega \subset \prod\limits_{i=1}^n X_i$, $\Omega' \subset \prod\limits_{i=1}^n Y_i$ be measurable sets.
Define a product operator $K$ which acts on $(\mu_1 \times \dots \times \mu_n)$-measurable functions as follows 
$$
(Kf)(x) = \int\limits_{\Omega'} \prod\limits_{i=1}^n k_i(x_1, \dots, x_i, y_i)f(y)\, d(\nu_1 \times \dots \times\nu_n),
$$
where $k_i$ are non-negative measurable functions. 
Next, define partial operators $K_i$
$$
(K_ig_i)(y_1, \dots, y_{i-1}, x_i) = \int\limits_{\pi_i(\Omega')} k_i(x_1, \dots, x_i, y_i)g(y_i) \, d\nu_i,
$$ 
By the Tonelli theorem
\begin{multline*}
(Kf)(x) = \int\limits_{Y_1}k_1(x_1,y_1)\int\limits_{Y_2} \dots \int\limits_{Y_n} k_n(x_1, \dots, x_n, y_n) f(y) \chi_{\Omega'}(y) \,d\nu_n \dots d\nu_2 d\nu_1 \\
= \int\limits_{Y_1}k_1(x_1,y_1) K_{x_1}f_{y_1}(\tilde y)\,d\nu_1. 
\end{multline*}
where $K_{x_1}$ is a product operator with $n-1$ kernels $k_i = k_i(x_1, x_2, \dots, x_i, y_i)$, $i = 2, \dots, n$ which acts on functions $f_{y_1}(\tilde y)$ by the rule
$$
(K_{x_1}f_{y_1})(\tilde x) = \int\limits_{\Omega'_{y_1}} \prod\limits_{i=2}^n k_i(x_1, x_2 \dots, x_i, y_i)f_{y_1}(\tilde y) \, d(\nu_2 \times \dots \times\nu_n).
$$

Here, we find conditions when the operator
$$
K : L_Q(\Omega') \to L_P(\Omega)
$$
is bounded. 
In the following theorem, sufficient conditions are given
in terms of boundedness of the partial operators.
The proof partly follows the approach from \cite{AK}.

\begin{theorem}\label{theorem:product}
If for $i = 1, \dots, n$
\begin{equation}\label{theorem:product:eq1}
\|K_i g_i(y_1, \dots, y_{i-1}, \cdot) \|_{L_{p_i}(\pi_i(\Omega))} \leq C_i\|g_i(y_1, \dots, y_{i-1}, \cdot) \|_{L_{q_i}(\pi_i(\Omega'))}
\end{equation}
$\nu_1 \times \dots \times \nu_{i-1}$-a.e. and for all $g_i(y_1, \dots, y_{i-1}, \cdot) \in L_{q_i}(\pi_i(\Omega'))$.

Then 
$$
\|Kf\|_{L_P(\Omega)} \leq \prod\limits_{i=1}^n C_i \|f\|_{L_Q(\Omega')}.
$$
\end{theorem}

\begin{proof}
Again use the principle of mathematical induction to prove the theorem.
The base case of induction folloving from theorem's assertions 
(that the product operator $K_{x_1, \dots, x_{n-1}}$ coincides with the one-dimensional operator $K_n$).
Suppose that for some $n$ theorem is proved. Showing that the same statement also holds for $n+1$.
Using Minkowski's integral inequality \eqref{eq:minkowski} $n$ times
\begin{align*}
\|Kf\|_{L_{P}(\Omega)} = \bigg(\int\limits_{X_1}\bigg(\dots & \bigg(\int\limits_{X_{n+1}} \bigg[ \chi_{\Omega}(x)\int\limits_{Y_1} k_1(x_1,y_1) \times\\
&\times K_{x_1}f_{y_1}(\tilde x)\, d\nu_1 \bigg]^{p_{n+1}} \, d\mu_{n+1}\bigg)^{\frac{p_n}{p_{n+1}}} \dots \bigg)^{\frac{p_1}{p_2}} \, d\mu_1 \bigg)^{\frac{1}{p_1}}\\
\leq \bigg(\int\limits_{X_1} \bigg(\int\limits_{Y_1} k_1(x_1,y_1) & \bigg[\int\limits_{X_2} \bigg( \dots \bigg(\int\limits_{X_{n+1}} (K_{x_1}f_{y_1}(\tilde x))^{p_{n+1}} \times \\
&\times \chi_{\Omega_{x_1}}(\tilde x) \, d\mu_{n+1} \bigg]^{\frac{p_n}{p_{n+1}}} \dots \bigg)^{\frac{p_2}{p_3}} d\mu_2\bigg)^{\frac{1}{p_2}} d\nu_1\bigg)^{p_1}  d\mu_1 \bigg)^{\frac{1}{p_1}}\\
= \bigg(\int\limits_{X_1} \bigg(\int\limits_{Y_1}k_1(x_1,y_1) & \|K_{x_1}f_{y_1}\|_{L_{\widetilde {P}}(\Omega_{x_1})} \, d\nu_1\bigg)^{p_1}\,d\mu_1 \bigg)^{\frac{1}{p_1}}.
\end{align*}
Note that if $f\in L_Q(\Omega')$ then the function $h(y_1) = \|f_{y_1}\|_{L_{\widetilde{Q}}({\Omega'}_{y_1})}$ belongs to $L_{q_1}(Y_1)$.
Then, we apply \eqref{theorem:product:eq1} along with the induction hypothesis and deduce
\begin{multline*}
\bigg(\int\limits_{X_1} \bigg(\int\limits_{Y_1}k_1(x_1,y_1) \|K_{x_1}f_{y_1}\|_{L_{\widetilde {P}}(\Omega_{x_1})} \, d\nu_1 \bigg)^{p_1}\,d\mu_1 \bigg)^{\frac{1}{p_1}}\\
\leq \prod\limits_{i=2}^{n+1}C_i \bigg(\int\limits_{X_1} \bigg(\int\limits_{Y_1}k_1(x_1,y_1) \|f_{y_1}\|_{L_{\widetilde {q}}(\Omega'_{y_1})} \, d\nu_1 \bigg)^{p_1}\,d\mu_1 \bigg)^{\frac{1}{p_1}}\\ 
= \prod\limits_{i=2}^{n+1}C_i \bigg(\int\limits_{X_1} \bigg((K_1 h)(x_1)\bigg)^{p_1}\, d\mu_1 \bigg)^{\frac{1}{p_1}}
\leq \prod\limits_{i=1}^{n+1}C_i \bigg(\int\limits_{Y_1} |h(y_1)|^{q_1}\, d\nu_1 \bigg)^{\frac{1}{q_1}} \\
= \prod\limits_{i=1}^{n+1}C_i\|f\|_{L_Q(\Omega')}
\end{multline*}
as desired.
\end{proof}

If $\Omega=\Omega'=\mathbf R^2_{+}$ and $k_1(x_1,y_1) = \chi_{[0,x_1]}(y_1)$, $k_2(x_1,x_2,y_2) = \chi_{[0,x_2]}(y_2)$ then 
the product operator turns out to be a two-dimensional Hardy operator 
$$
Hf = \int\limits_0^{x_1}\int\limits_0^{x_2} f(y_1,y_2)\,dy_2dy_1,
$$
which acts from $L_{q_1,q_2}(\mathbf R^2, \nu_1, \nu_2)$ to $L_{p_1,p2}(\mathbf R^2, \mu_1, \mu_2)$.

We derive from theorem \ref{theorem:product} that sufficient conditions for the boundedness of $H$ on mixed norm Lebesgue spaces 
are the boundedness of two one-dimensional Hardy operators   
$$
H_1g = \int\limits_0^{x_1} g(y_1)\,dy_1 \quad\text{ and }\quad H_1g = \int\limits_0^{x_2} g(y_2)\,dy_2. 
$$
on corresponding Lebesgue spaces. 
The above statement agrees with \cite[Proposition 2]{AK} when measures $\mu_i$, $\nu_i$ are weighted functions.

\section{Hardy-Steklov type operators}
In this section we apply the technique from the previous one to revise properties of multiplication operators.
Besides we obtain a corollary of theorem \ref{theorem:boundedLpq} for Hardy-Steklov type operators.

\subsection{Multiplication operator}
Let $L_{P}(\Omega)$ be as in section \ref{composition}.
Define a multiplication operator $M_g : L_{P}(\Omega) \to L_{P}(\Omega)$ by the rule $(M_g f)(x) = f(x)g(x)$. 
Applying considerations similar to the ones in the proof of theorem \ref{theorem:boundedLpq}, we obtain the following proposition.

\begin{theorem}\label{theorem:Multiplication}
The multiplication operator $M_g$ is bounded if and only if $g \in L_\infty(\Omega)$.
The norm of the operator $ \|M_g\| = \esssup\limits_{\Omega} g(x) $.
\end{theorem}

\begin{proof}
It is well known that the theorem holds in the case of Lebesgue space $L_p(\Omega)$ (i.e.\ $n=1$). 
Thus the base case is proven.

Suppose the assertion of the theorem is valid for some $n$.

Take any ball $B=B(\stackrel ox_1,\rho)\subset X_1$ and an arbitrary family $\{f_{x_1}(\tilde x)\}$ which consists
of such functions $f_{x_1}(\tilde x)\in L_{\widetilde P}(\Omega_{x_1})$ and 
$\|f_{x_1} \|_{L_{\widetilde P}(\Omega_{x_1})} = 1$. 
Then for function  $ h(x) = \chi_B(x_1)f_{x_1}(\tilde x) $, due to the boundedness of $M_g$ 
$$
\bigg(\int\limits_B \|g(x_1,\cdot)f_{x_1}(\cdot) \|_{L_{\widetilde P}(\Omega_{x_1})}^{p_1} \, d\mu_1\bigg)^{\frac{1}{p_1}} \leq K (\mu_1(B))^{\frac{1}{p_1}}.
$$
Further, taking the supremum over all $f_{x_1}(\tilde x)$ one obtains the norm of multiplication operator 
$(M_{g(x_1, \cdot)}f_{x_1})(\tilde x) = f_{x_1}(\tilde x)g(x_1,\tilde x)$ under the integral. 
By the induction assumption  it equals $\|g(x_1,\cdot)\|_{L_\infty(\Omega_{x_1})}$.

$$
\bigg(\int\limits_B \|g(x_1,\cdot)\|^{p_1}_{L_\infty(\Omega_{x_1})} \, d\mu_1\bigg)^{\frac{1}{p_1}} \leq K (\mu_1(B))^{\frac{1}{p_1}}.
$$

Applying Lebesgue's theorem we conclude 
$$
\|g(\stackrel ox_1,\cdot)\|^{p_1}_{L_\infty(\Omega_{\stackrel ox_1})} \leq K
$$
and then, due to arbitrariness of the choice of ball $B$, conclude that 
$$
\esssup\limits_{\Omega}g(x) \leq K.
$$
Thus, we have completed the induction step. 
\end{proof}

\subsection{Hardy-Steklov type operators}

Consider another operator $I : L_P(\prod\limits_{i=1}^n Y_i) \to L_P(\Omega)$ by the rule:
$$
(I f)(x) = \frac{1}{\psi_1(x_1)\cdot\dots\cdot\psi_n(x)}\int\limits_{0}\limits^{\psi_1(x_1)} \dots \int\limits_{0}\limits^{\psi_n(x)} f(y) g(y) \, dy_n \dots \, dy_1,
$$
where mappings $\psi_i$ and function $g$ are defined above.

By using the approach from the proof of theorem \ref{theorem:boundedLpq}  we can
prove the following theorem.

\begin{theorem}\label{theorem:CHMg}
Operator $I$ is bounded if and only if 
$$
\esssup\limits_{\Omega'} g(y) \bigg\{\prod\limits_{i=1}^n \frac{p_i}{p_i-1} J(\psi^{-1}_i(y_1,\dots,y_{i-1},\cdot); y_i)^{\frac{1}{p_i}}\bigg\}  < \infty.
$$
The norm of the operator 
$$
\| I \| = \esssup\limits_{\Omega'} g(y) \bigg\{\prod\limits_{i=1}^n \frac{p_i}{p_i-1} J(\psi^{-1}_i(y_1,\dots,y_{i-1},\cdot); y_i)^{\frac{1}{p_i}}\bigg\}.
$$
\end{theorem}
Operator $I$ has a representation as a product $C_{\varphi}H_nM_g$, assuming all involved operators are bounded. 

Finally, we write down sufficient conditions under which the two-dimensional Hardy-Steklov type operator
$$
\mathcal K f(x_1,x_2) =  \int\limits_{a_1(x_1)}^{b_1(x_2)}\int\limits_{a_2(x_1,x_2)}^{b_2(x_1,x_2)} f(y_1,y_2)k_1(x_1,y_1)k_2(x_1,x_2,y_2)\, dy_2 \, dy_1
$$
is bounded.
Derive from theorem \ref{theorem:product} that $\mathcal K : L_{q_1,q_2}(\mathbf R^2, \nu_1, \nu_2) \to L_{p_1,p_2}(\mathbf R^2, \mu_1, \mu_2)$  is bounded if the following two one-dimensional Hardy-Steklov type operators
$$
\mathcal K_1 g(x_1) =  \int\limits_{a_1(x_1)}^{b_1(x_1)}g(y_1)k_1(x_1,y_1) \, dy_1 \quad\text{ and }\quad 
\mathcal K_2 g(x_2) =  \int\limits_{a_2(x_1,x_2)}^{b_2(x_1,x_2)} g(y_2)k_2(x_1,x_2,y_2)\, dy_2
$$
are bounded in corresponding Lebesgue spaces. 
Conditions for the boundedness of one-dimensional Hardy-Steklov type operator can be found in \cite{SU, BJT}.

%

\providecommand{\WileyBibTextsc}{}
\let\textsc\WileyBibTextsc
\providecommand{\othercit}{}
\providecommand{\jr}[1]{#1}
\providecommand{\etal}{~et~al.}

\end{document}